\documentclass[a4paper]{elsarticle}

\usepackage{amsmath,amssymb,amsthm,bm}
\usepackage{enumerate}
\theoremstyle{plain}
   \newtheorem{theorem}{Theorem}[section]
   \newtheorem{proposition}[theorem]{Proposition}
   \newtheorem{corollary}[theorem]{Corollary}
   
\theoremstyle{definition}
   \newtheorem{definition}[theorem]{Definition}
   
   \newtheorem{example}[theorem]{Example}
\theoremstyle{remark}
   \newtheorem{remark}[theorem]{Remark}
\newcommand{\calA}{\mathcal{A}}
\newcommand{\calB}{\mathcal{B}}
\newcommand{\calD}{\mathcal{D}}
\newcommand{\calE}{\mathcal{E}}
\newcommand{\calF}{\mathcal{F}}

\newcommand{\calLpq}{\mathcal{L}^{p,q}}

\newcommand{\calM}{\mathcal{M}}
\newcommand{\calN}{\mathcal{N}}
\newcommand{\calS}{\mathcal{S}}
\newcommand{\calT}{\mathcal{T}}
\newcommand{\calU}{\mathcal{U}}
\newcommand{\bbN}{\mathbb{N}}
\newcommand{\bbR}{\mathbb{R}}
\newcommand{\exR}{\overline{\bbR}}

\newcommand{\seql}{{l\in\bbN}}

\newcommand{\seqn}{{n\in\bbN}}
\newcommand{\seqk}{{k\in\bbN}}
\newcommand{\ninfty}{{n\to\infty}}

\newcommand{\ep}{\varepsilon}
\newcommand{\eset}{\emptyset}
\newcommand{\td}{\widetilde{d}}

\newcommand{\tG}{\widetilde{G}}
\newcommand{\tS}{\widetilde{S}}
\newcommand{\tU}{\widetilde{U}}
\newcommand{\tV}{\widetilde{V}}

\newcommand{\tcalT}{\widetilde{\calT}}
\newcommand{\tcalU}{\widetilde{\calU}}

\newcommand{\mconv}{\stackrel{\!\mu}{\longrightarrow}}
\newcommand{\us}[1]{{\upshape #1}}

\newcommand{\pqn}[1]{\|#1\|_{p,q}}

\newcommand{\dsn}[1]{\|#1\|_{0}}

\newcommand{\norm}[1]{\|#1\|}

\newcommand{\nSu}[1]{(#1)_{1}}
\newcommand{\pqnSu}[1]{(#1)_{p,q}}

\newcommand{\mupq}{\mu^{q/p}}
\newcommand{\Ch}{\textup{Ch}}
\newcommand{\Su}{\textup{Su}}

\newcommand{\fSu}{\frak{S\hspace*{-0.2mm}u}}

\newcommand{\qSu}{\mbox{\textit{S\hspace*{-0.1mm}u}}}

\newcommand{\LS}{\left(\frac{p}{q}\right)^\frac{1}{q}}

\newcommand{\ls}{(p/q)^{1/q}}
\newcommand{\rLS}{\left(\frac{q}{p}\right)^\frac{1}{q}}

\begin{document}
\allowdisplaybreaks
\begin{frontmatter}
%
%
\title{Topological and topological linear properties of the
Sugeno-Lorentz spaces\tnoteref{t1}}
\tnotetext[t1]{This work was supported by JSPS KAKENHI Grant Number 20K03695.}
%
%
\author{Jun Kawabe\corref{cor1}}
\ead{jkawabe@shinshu-u.ac.jp}
\address{Faculty of Engineering, Shinshu University, 4-17-1 Wakasato, Nagano 380-8553, Japan}
\cortext[cor1]{Corresponding author}
%
%
\begin{abstract}
The properties of spaces of Sugeno integrable functions are quite different from those of
ordinary spaces of Lebesgue integrable functions.
In our previous research the completeness and separability of the Sugeno-Lorentz spaces
were discussed in terms of the characteristic of nonadditive measures
without a detailed study of their topologies.
The purpose of the paper is to further advance our study of the Sugeno-Lorentz spaces
by investigating some fundamental topological and topological linear properties
of the Sugeno-Lorentz spaces.
\end{abstract}
%
%
\end{frontmatter}


%
%
\section{Introduction}\label{intro}
The Sugeno integral was introduced by Sugeno~\cite{Sugeno} under the name of fuzzy integral
and has many important and interesting applications; see~\cite{Chitescu,G-M-S,Grabisch,W-K}.
\par
Given any nonadditive measure $\mu$ on a measurable space $(X,\calA)$
and any Orlicz function $\varPhi$,
the Orlicz space $L_\varPhi(\mu)$
and $\varPhi$-mean convergence were defined in~\cite{W-M}
by using the Sugeno integral.
In particular, if $\varPhi(t):=t^p$, where $0<p<\infty$, then the space $L_\varPhi(\mu)$
are reduced to the space $L^p(\mu)$ consisting of all $\calA$-measurable real-valued functions on $X$
such that $|f|^p$ are Sugeno integrable with respect to $\mu$.
Then, as a particular case of~\cite[Theorem~2]{W-M} it follows that
$L^p(\mu)=L^1(\mu)$ and that $p$-th order convergence is equivalent
to convergence in $\mu$-measure.
This observation suggests that the properties of spaces of Sugeno integrable functions
are quite different form those of ordinary spaces of Lebesgue integrable functions.  
\par
In~\cite{K2021b}, for any $0<p<\infty$ and $0<q<\infty$,
the Sugeno-Lorentz prenorm $\pqnSu{\,\cdot\,}$
and the Sugeno-Lorentz space $\fSu(\mu)$ were defined as the space of
all $\calA$-measurable functions on $X$ such that $|f|^q$ are Sugeno integrable with respect to
the nonadditive measure $\mupq$.
Our definition of the Sugeno-Lorentz spaces and the Sugeno-Lorentz prenorms
agrees with ordinary Lorentz spaces and Lorentz norms~\cite{Lorentz1950,Lorentz1953}
if $\mu$ is $\sigma$-additive and the Sugeno integral is replaced with the Lebesgue integral;
see~\cite{K2021b}.
As a starting point of research,
the completeness and separability of the Sugeno-Lorentz spaces were discussed
in~\cite{K2021b} in terms of the characteristics of nonadditive measures without a detailed
study of their topologies.
In this paper, we focus on the topological and topological linear properties
of the Sugeno-Lorentz spaces.
\par
This paper is organized as follows. Section~\ref{pre} sets up notation and terminology.
It also contains a discussion of an equivalence relation in the space of measurable functions
and the topology generated by a distance function.
In Section~\ref{S-space}, given a nonadditive measure $\mu$ on a measurable space $(X,\calA)$
the Sugeno-Lorentz space and the Sugeno-Lorentz prenorm $\pqnSu{\,\cdot\,}$
are defined for every $0<p<\infty$ and $0<q<\infty$.
This section also contains a summary of results of~\cite{K2021b} that will be used later.
Topological properties of the Sugeno-Lorentz spaces are discussed in Section~\ref{Top}.
First, the Sugeno-Lorentz topology
is defined as the topology generated by the prenorm $\pqnSu{\,\cdot\,}$
and it is shown that any sphere determined by $\pqnSu{\,\cdot\,}$ is open with respect to this topology
if and only if $\mu$ is autocontinuous from above.
It is also shown that the Sugeno-Lorentz topology
is the relative topology induced from the Dunford-Schwartz topology introduced in~\cite{K2021a}.
This fact implies that the Sugeno-Lorentz topology does not depend on $p$ and $q$,
and hence, it is only one regardless of $p$ and $q$.
In Section~\ref{TLS} some fundamental topological linear properties of the Sugeno-Lorentz spaces
are deduced from the corresponding properties of the space of all measurable functions with
the Dunford-Schwartz topology.
Section~\ref{conclusion} provides a summary of our results.
%
%
\section{Preliminaries}\label{pre}
Throughout the paper, $(X,\calA)$ is a measurable space, that is,
$X$ is a nonempty set and $\calA$ is a $\sigma$-field of subsets of $X$.
Let $\bbR$ denote the set of the real numbers and $\bbN$
the set of the natural numbers.
Let $\exR:=[-\infty,\infty]$ be the set of the extended real numbers with usual total order
and algebraic structure.
Assume that $(\pm\infty)\cdot 0=0\cdot (\pm\infty)=0$
since this proves to be convenient in measure and integration theory.
\par
For any $a,b\in\exR$, let $a\vee b:=\max\{a,b\}$ and $a\wedge b:=\min\{a,b\}$
and for any $f,g\colon X\to\exR$, let $(f\vee g)(x):=f(x)\vee g(x)$ and $(f\wedge g)(x):=f(x)\wedge g(x)$
for every $x\in X$.
Let $\calF_0(X)$ denote the set of all $\calA$-measurable real-valued functions on $X$.
Then $\calF_0(X)$ is a real linear space with usual pointwise addition and scalar multiplication.
For any $f,g\in\calF_0(X)$, the notation $f\leq g$ means that $f(x)\leq g(x)$ for every $x\in X$.
Let $\calF_0^+(X):=\{f\in\calF_0(X)\colon f\geq 0\}$.
A \emph{simple}\/ function is a function taking only a finite number of real numbers.
Let $\calS(X)$ denote the set of all $\calA$-measurable simple functions on $X$.
\par
For a sequence $\{a_n\}_\seqn\subset\exR$ and $a\in\exR$, the notation
$a_n\uparrow a$ means
that $\{a_n\}_\seqn$ is nondecreasing and $a_n\to a$, and $a_n\downarrow a$ means
that $\{a_n\}_\seqn$ is nonincreasing and $a_n\to a$. 
For a sequence $\{A_n\}_\seqn\subset\calA$ and $A\in\calA$, the notation $A_n\uparrow A$ means
that $\{A_n\}_\seqn$ is nondecreasing and $A=\bigcup_{n=1}^\infty A_n$, and $A_n\downarrow A$
means that $\{A_n\}_\seqn$ is nonincreasing and $A=\bigcap_{n=1}^\infty A_n$.
The characteristic function of a set $A$, denoted by $\chi_A$, is the function on $X$
such that $\chi_A(x)=1$ if $x\in A$ and $\chi_A(x)=0$ otherwise.
Given two sets $A$ and $B$, let $A\triangle B:=(A\setminus B)\cup (B\setminus A)$
and $A^c:=X\setminus A$.
Let $2^X$ denote the collection of all subsets of $X$.
%
%
\subsection{Nonadditive measures}\label{measure}
A \emph{nonadditive measure}\/ is a set function $\mu\colon\calA\to [0,\infty]$ such that
$\mu(\emptyset)=0$ and $\mu(A)\leq\mu(B)$ whenever $A,B\in\calA$ and $A\subset B$.
This type of set function is also called a monotone measure~\cite{W-K}, a capacity~\cite{Choquet},
or a fuzzy measure~\cite{R-A,Sugeno} in the literature.
Let $\calM(X)$ denote the set of all nonadditive measures $\mu\colon\calA\to [0,\infty]$.
\par
Let $\mu\in\calM(X)$.
We say that $\mu$ is \emph{order continuous}~\cite{Drewnowski}
if $\mu(A_n)\to 0$ whenever $A_n\downarrow\eset$,
\emph{continuous from above}\/ if $\mu(A_n)\to\mu(A)$ whenever $A_n\downarrow A$,
\emph{continuous from below}\/ if $\mu(A_n)\to\mu(A)$ whenever $A_n\uparrow A$,
and \emph{null-continuous}~\cite{U-M}\/ if $\mu(\bigcup_{n=1}^\infty N_n)=0$
whenever $\{N_n\}_\seqn\subset\calA$ is nondecreasing and $\mu(N_n)=0$ for every $\seqn$.
The order continuity follows from the continuity form above, while the null-continuity
follows from the continuity from below.
\par
Following the terminology used in~\cite{W-K},
$\mu$ is called \emph{weakly null-additive}\/ if $\mu(A\cup B)=0$
whenever $A,B\in\calA$ and $\mu(A)=\mu(B)=0$,
\emph{null-additive}\/ if $\mu(A\cup B)=\mu(A)$ whenever $A,B\in\calA$ and $\mu(B)=0$,
\emph{autocontinuous from above}\/ if $\mu(A\cup B_n)\to\mu(A)$ whenever $A,B_n\in\calA$
and $\mu(B_n)\to 0$, and \emph{autocontinuous from below}\/ if $\mu(A\setminus B_n)\to\mu(A)$
whenever $A,B_n\in\calA$ and $\mu(B_n)\to 0$.
Furthermore, we say that $\mu$ satisfies
the \emph{pseudometric generating property}\/ ((p.g.p.)~for short)~\cite{D-F}
if $\mu(A_n\cup B_n)\to 0$ whenever $A_n,B_n\in\calA$
and $\mu(A_n)\lor\mu(B_n)\to 0$.
Every nonadditive measure satisfying the (p.g.p.) is weakly null-additive.
If $\mu$ is autocontinuous from above or below, then it is null-additive, hence weakly null-additive.
\par
A nonadditive measure $\mu$ is called \emph{subadditive}\/ if 
$\mu(A\cup B)\leq\mu(A)+\mu(B)$ for every disjoint $A,B\in\calA$,
\emph{relaxed subadditive}\/ if there is a constant $K\geq 1$ such that
$\mu(A\cup B)\leq K\left\{\mu(A)+\mu(B)\right\}$ for every disjoint $A,B\in\calA$ (in this case
$\mu$ is called \emph{$K$-relaxed subadditive}).
Every subadditive nonadditive measure is relaxed subadditive.
If $\mu$ is relaxed subadditive, then it satisfies the (p.g.p.).
See ~\cite{Denneberg,Pap,W-K} for further information on nonadditive measures.
\begin{remark}
The relaxed subadditivity is also called the quasi-subadditivity according to the terminology
used in metric space theory.
\end{remark}
\subsection{The Sugeno integral}\label{integral}
The Sugeno integral is one of important integrals that is widely used in nonadditive measure theory
and its applications.
The \emph{Sugeno integral}~\cite{R-A,Sugeno} is defined by
\[
\Su(\mu,f):=\sup_{t\in [0,\infty)}t\land\mu(\{f>t\})
\]
for every $\mu\in\calM(X)$ and $f\in\calF_0^+(X)$.
In the above definition the nonincreasing distribution function $\mu(\{f>t\})$,
where $\{f>t\}:=\{x\in X\colon f(x)>t\}$,
may be replaced with $\mu(\{f\geq t\})$ and the interval of the range in which the variable $t$ moves
may be replaced with $[0,\infty]$ or $(0,\infty)$ without changing the integral value.
\par
The following properties of the Sugeno integral are 
easy to prove and used without explicitly mentioning; see also~\cite{W-K}.
Let $f,g\in\calF_0^+(X)$, $A\in\calA$, $\alpha\geq 0$, and $0<p<\infty$.

\begin{itemize}
\item Monotonicity:\ If $f\leq g$ then $\Su(\mu,f)\leq\Su(\mu,g)$.
\item Generativity:\ $\Su(\mu,\alpha\chi_A)=\alpha\land\mu(A)$.
\item Truncated subhomogeneity:\ $\Su(\mu,\alpha f)\leq\max\{1,\alpha\}\Su(\mu,f)$.
\item Exponentiation:\ $\Su(\mu,f^p)=\Su(\mu^{1/p},f)^p$.
\item Integrability:\ $\Su(\mu,f)\leq t\lor\mu(\{f>t\})$ for every $t\in [0,\infty]$,
so that
$\Su(\mu,f)<\infty$ if and only if
there is $t_0\in [0,\infty)$ such that $\mu(\{f>t_0\})<\infty$.
\end{itemize}
\par
Note that the Sugeno integral is neither additive nor positively homogeneous in general.
The relaxed subadditivity of the Sugeno integral
can be characterized in terms of nonadditive measures; see~\cite[Proposition~2.2]{K2021b}.
\begin{proposition}\label{K-sub}
Let $\mu\in\calM(X)$.
The following conditions are equivalent.
\begin{enumerate}
\item[\us{(i)}] $\mu$ is $K$-relaxed subadditive for some $K\geq 1$.
\item[\us{(ii)}] The Sugeno integral is $K$-relaxed subadditive, that is,
for any $f,g\in\calF_0^+(X)$ it follows that
\[
\Su(\mu,f+g)\leq K\bigl\{\Su(\mu,f)+\Su(\mu,g)\bigr\}.
\]
\end{enumerate}
In particular, the Sugeno integral is subadditive if and only if $\mu$ is subadditive.
\end{proposition}
\begin{remark}
The subadditivity of the Sugeno integral was proved in~\cite[Proposition~5]{G-M-R-S}
for nonadditive measures on the discrete space $(\bbN,2^\bbN)$.
\end{remark}
\subsection{Convergence in measure of measurable functions}\label{mode}

The concept of convergence in measure is not quite intuitive, but it has some advantages in analysis.
Let $\{f_n\}_\seqn\subset\calF_0(X)$ and $f\in\calF_0(X)$.
We say that $\{f_n\}_\seqn$ converges \emph{in $\mu$-measure}\/ to $f$, denoted by $f_n\mconv f$,
if $\mu(\{|f_n-f|>\ep\})\to 0$ for every $\ep>0$.
This mode of convergence requires that the differences between the elements $f_n$
of the sequence and the limit function $f$ should become small in some sense as $n$ increases.
The following definition involves only the elements of the sequence.
We say that $\{f_n\}_\seqn$
is \emph{Cauchy in $\mu$-measure}\/ if for any $\ep>0$ and $\delta>0$
there is $n_0\in\bbN$ such that $\mu(\{|f_m-f_n|>\ep\})<\delta$
whenever $m,n\in\bbN$ and $m,n\geq n_0$.
\par
See a survey paper~\cite{L-M-P-K} for further information on various modes
of convergence of measurable functions in nonadditive measure theory.
%
%
\subsection{Equivalence relation and quotient space}\label{equiv}
The quotient space of $\calF_0(X)$ is constructed by an equivalence relation
determined by a nonadditive measure $\mu$.
The proof of the following statements is routine and left it to the reader.
\begin{itemize}
\item Assume that $\mu$ is weakly null-additive. Given $f,g\in\calF_0(X)$,
define the binary relation $f\sim g$ on $\calF_0(X)$ by $\mu(\{|f-g|>c\})=0$ for every $c>0$
so as to become an equivalence relation on $\calF_0(X)$.
For every $f\in\calF_0(X)$ the equivalence class of $f$
is the set of the form $\{g\in\calF_0(X)\colon f\sim g\}$ and denoted by $[f]$.
Then the quotient space of $\calF_0(X)$ is defined by $F_0(X):=\{[f]\colon f\in\calF_0(X)\}$
and the mapping $\kappa\colon\calF_0(X)\to F_0(X)$ defined by $\kappa(f):=[f]$
for every $f\in\calF_0(X)$ is known as the quotient map.
\item Assume that $\mu$ is weakly null-additive.
Given equivalence classes $[f],[g]\in F_0(X)$ and $\alpha\in\bbR$, define addition and scalar multiplication
on $F_0(X)$ by $[f]+[g]:=[f+g]$ and $\alpha [f]:=[\alpha f]$.
They are well-defined, that is, they are independent of which member of an equivalence class we choose
to define them.
Then $F_0(X)$ is a real linear space.
\end{itemize}
\par
The space $F_0(X)$ is exactly the same as the quotient space $\calF_0(X)/\calN$, where
$\calN:=\{f\in\calF_0(X)\colon\mu(\{|f|>c\})=0\mbox{ for every }c>0\}$.
The binary relation on $\calF_0(X)$ defined above may not
be transitive unless $\mu$ is weakly null-additive; see~\cite[Example~5.1]{K2021a}.
In what follows, let $S(X):=\{[h]\colon h\in\calS(X)\}$.
%
%
\subsection{Prenorms}\label{prenorm}
Let $V$ be a real linear space.
A \emph{prenorm} on $V$ is a nonnegative real-valued
function $\norm{\cdot}$
defined on $V$ such that $\norm{0}=0$ and $\norm{\!-\!x}=\norm{x}$ for every $x\in V$.
Then the pair $(V,\norm{\cdot})$ is called a \emph{prenormed space}.
A prenorm $\norm{\cdot}$ is called \emph{homogeneous}\/ if it follows that
$\norm{\alpha x}=|\alpha|\norm{x}$ for every $x\in V$ and $\alpha\in\bbR$ and
\emph{truncated subhomogeneous}\/ if it follows that
$\norm{\alpha x}\leq\max(1,|\alpha|)\norm{x}$ for every $x\in V$ and $\alpha\in\bbR$.
Following~\cite{D-D}, a prenorm $\norm{\cdot}$ is called relaxed if it
satisfies a \emph{relaxed triangle inequality}, that is, there is a constant $K\geq 1$
such that $\norm{x+y}\leq K\left\{\norm{x}+\norm{y}\right\}$ for every $x,y\in V$
(in this case, we say that $\norm{\cdot}$ satisfies $K$-relaxed triangle inequality).
\par
To associate with similar characteristics of nonadditive measures,
a prenorm $\norm{\cdot}$ is called \emph{weakly null-additive}\/ if $\norm{x+y}=0$
whenever $x,y\in V$ and $\norm{x}=\norm{y}=0$ and
\emph{null-additive}\/ if $\norm{x+y}=\norm{x}$
whenever $x,y\in V$ and $\norm{y}=0$.
\par
Let $(V,\norm{\cdot})$ be a prenormed space.
Let $\{x_n\}_\seqn\subset V$ and $x\in V$.
We say that $\{x_n\}_\seqn$ \emph{converges}\/ to $x$, denoted by $x_n\to x$,
if $\norm{x_n-x}\to 0$.
We also say that $\{x_n\}_\seqn$ is \emph{Cauchy}\/ if for
any $\varepsilon>0$ there is $n_0\in\bbN$
such that $\norm{x_m-x_n}<\varepsilon$ whenever $m,n\in\bbN$ and $m,n\geq n_0$.
Not every converging sequence is Cauchy since prenorms satisfy neither
the triangle inequality nor its relaxed ones in general.
A subset $B$ of $V$ is called \emph{bounded}\/ if $\sup_{x\in B}\norm{x}<\infty$.
If the prenorm is relaxed, then every converging sequence is Cauchy
and every Cauchy sequence is bounded.
\par
A prenormed space $(V,\norm{\cdot})$ is called \emph{complete}\/
if every Cauchy sequence in $V$ converges to an element in $V$.
It is called \emph{quasi-complete}\/ if every bounded Cauchy sequence
in $V$ converges to an element in $V$.
The denseness and the separability can be defined in the same way as in ordinary normed spaces.
We say that $V$ is \emph{separable}\/ if there is a countable subset $D$ of $V$ such that
$D$ is \emph{dense} in $V$, that is,
for any $x\in V$ and $\ep>0$ there is $y\in D$ for which $\norm{x-y}<\ep$.
\par
In the above terms, if we want to emphasize that we are thinking of $\norm{\cdot}$ as a prenorm,
then the phrase ``with respect to $\norm{\cdot}$'' is added to each term.
%
%
\subsection{Topological concepts}\label{SS:topology}
The topological concepts used in this paper are standard and can be found in~\cite{Willard}.
It is the custom of a great many mathematicians to use ``neighborhood of $x$'' to mean
``open neighborhood of $x$.''
In this paper, neighborhoods are not necessarily open, however, unless so described.
Given a sequence $\{x_n\}_\seqn\subset T$ and $x\in T$ in a topological space $(T,\calT)$,
the notation $x_n\to x$ means that $\{x_n\}_\seqn$ converges to $x$
with respect to the topology $\calT$.
\subsection{Distance functions and distance topologies}\label{Pmetric}
A \emph{distance function} (\emph{distance} for short)
on $T$ is a nonnegative real-valued function $d$ defined on $T\times T$ such that
$d(x,x)=0$ and $d(x,y)=d(y,x)$ for every $x,y\in T$;
see for instance~\cite[p.~3]{D-D} and~\cite[p.~16]{Willard}.
It is called a \emph{pseudometric}\/ if it satisfies the triangle inequality,
that is, $d(x,y)\leq d(x,z)+d(z,y)$ for every $x,y,z\in T$.
A pseudometric is a \emph{metric} if it separates points of $T$,
that is, for any pair of points $x,y\in T$, if $d(x,y)=0$ then $x=y$.
\par
Given a distance $d$ on $T$, one can construct a topology on $T$ in the following way.
For each $x\in T$ and $r>0$, let the sphere centered at $x$ with radius $r$ be
the subset of $T$ defined by
\[
S(x,r):=\{y\in T\colon d(x,y)<r\}.
\]
For each $x\in T$, let $\calU(x)$ be the collection of all subsets $U$ of $X$ such that
$S(x,r)\subset U$ for some $r>0$.
Let $\calT$ be the collection of all subsets $G$ of $T$ such that $G\in\calU(x)$ for every $x\in G$.
It is obvious that $\calT$ is the collection of all subsets $G$ of $T$
with the property that for any $x\in G$ there is $r>0$ such that $S(x,r)\subset G$.
%
%
%
\begin{theorem}\label{Top_property}
The collection $\calT$ is a topology on $T$ and has the following properties.
\begin{enumerate}
\item[\us{(1)}] For any $x\in T$, the neighborhood system of $x$
is contained in $\calU(x)$.
\item[\us{(2)}] For any sequence $\{x_n\}_\seqn\subset T$ and $x\in T$, 
if $d(x_n,x)\to 0$ then $x_n\to x$.
\item[\us{(3)}] The following conditions are equivalent.
\begin{enumerate}
\item[\us{(i)}] For any $x\in T$ and $r>0$, the sphere $S(x,r)$ is a neighborhood of $f$.
\item[\us{(ii)}] For any $x\in T$, the neighborhood system of $x$ coincides with $\calU(x)$.
\item[\us{(iii)}] The topology $\calT$ is first countable and for any sequence $\{x_n\}_\seqn\subset T$
and $x\in T$, it follows that $x_n\to x$ if and only if $d(x_n,x)\to 0$.
\end{enumerate}
\end{enumerate}
\end{theorem}
\begin{proof}
	The proof of (1), (2) and implication (i)$\Rightarrow$(ii) of (3)
    is the same as~\cite[Theorem~4.1]{K2021a}.
	\par
	(ii)$\Rightarrow$(iii)\ The first countability of $\calT$ follows from the fact that for each $x\in T$
	the collection $\calB(x):=\{S(x,1/n)\colon\seqn\}$
	is a neighborhood base of $x$.
	The proof of the rest is easy.
	\par
	(iii)$\Rightarrow$(i)\ Let $x\in T$ and $r>0$. Let $\{U_n\}_\seqn$ be a neighborhood base of $x$.
	Without loss of generality we may assume that $\{U_n\}_\seqn$ is nonincreasing.
	Suppose, contrary to our claim, that $S(x,r)$ is not a neighborhood of $x$.
	Then, for any $\seqn$, we have $U_n\not\subset S(x,r)$, for if not, $U_n\subset S(x,r)$,
	which is impossible since $S(x,r)$ is not a neighborhood of $x$.
	Form this, for each $\seqn$, we can take $x_n\in U_n$ such that $x_n\not\in S(x,r)$.
	Then $x_n\to x$, hence $d(x_n,x)\to 0$ by assumption.
	This leads to a contradiction since $d(x_n,x)\geq r$ for every $\seqn$.
\end{proof}

This way of constructing a topology is the same as that of constructing the usual metric topology,
but it should be noted that $\calU(x)$ is not always the neighborhood system of $x$.
The topology constructed in the way described above is called the
\emph{distance topology}\/ generated by a distance $d$.
A topological space $(T,\calT)$ is called \emph{pseudometrizable}\/ if the space $T$
can be given a pseudometric $d$
such that the distance topology generated by $d$ coincides with the given topology $\calT$.
In this case, $\calT$ is called \emph{pseudometrizable}\/ by the pseudometric $d$.
A \emph{metrizable}\/ topological space can be defined in the same way.
\par
Finally, a distance $d$ on a real linear space $V$ is called \emph{translation-invariant}\/ if
$d(x,y)=d(x+z,y+z)$ for every $x,y,z\in V$ and \emph{truncated subhomogeneous}\/ if
$d(\alpha x,\alpha y)\leq\max\{1,|\alpha|\}d(x,y)$ for every $x,y\in V$ and $\alpha\in\bbR$.
%
%
%
\section{The Sugeno-Lorentz spaces}\label{S-space}
Let $\mu\in\calM(X)$.
Let $\fSu(\mu)$ denote the set of all Sugeno integrable functions on $X$, that is,
\[
\fSu(\mu):=\{f\in\calF_0(X)\colon\Su(\mu,|f|)<\infty\}.
\]
In this section we define a counterpart of ordinary Lorentz spaces by using the Sugeno integral.
Let $\mu\in\calM(X)$.
Let $0<p<\infty$ and $0<q<\infty$.
When $\mu$ is $\sigma$-additive, the Lorentz space is defined by
\[
\calLpq(\mu):=\{f\in\calF_0(X)\colon\pqn{f}<\infty\},
\]
where $\pqn{f}$ is the Lorentz quasi-seminorm on $\calLpq(\mu)$ defined by the Lebesgue integral as
\begin{equation}\label{Lorentz}
	\pqn{f}:=p^{1/q}\left(\int_0^\infty\left[t\mu(\{|f|>t\})^{1/p}\right]^q\frac{dt}{t}\right)^{1/q}
\end{equation}
for every $f\in\calF_0(X)$~\cite[Theorem~6.6]{C-R}.
The right side of~\eqref{Lorentz} can be expressed as
\[
\Ch(\mupq,p|f|^q/q)^{1/q}
\]
in terms of the Choquet integral defined by
\[
\Ch(\mu,f):=\int_0^\infty\mu(\{f>t\})dt
\]
for every $f\in\calF_0^+(X)$ and $\mu\in\calM(X)$.
This rewriting of equation follows from the fact that
\[
\Ch(\mu,|f|^q)=\int_0^\infty qt^{q-1}\mu(\{|f|>t\})dt
\]
and the positive homogeneity of the Choquet integral, that is,
\[
\Ch(\mu,\alpha |f|)=\alpha\Ch(\mu,|f|)\quad (\alpha>0).
\]
This observation leads to the definition of the Sugeno-Lorentz prenorm.
\begin{definition}\label{Lprenorm}
	Let $\mu\in\calM(X)$. Let $0<p<\infty$ and $0<q<\infty$.
	Define the function $\pqnSu{\,\cdot\,}\colon\calF_0(X)\to [0,\infty]$ by
\begin{equation}\label{New}
\pqnSu{f}:=\Su(\mupq,p|f|^q/q)^{1/q}
\end{equation}
for every $f\in\calF_0(X)$.
\end{definition}
Since $\Su(\mu,|f|)<\infty$ if and only if $\Su(\mupq,p|f|^q/q)<\infty$ by~\cite[Lemma~9.4]{W-K},
it follows that
\[
\fSu(\mu)=\{f\in\calF_0(X)\colon\pqnSu{f}<\infty\}.
\]
Furthermore, if $\mu$ is finite, then $\Su(\mu,|f|)\leq\mu(X)<\infty$
for every $f\in\calF_0(X)$, hence we have
\[
\fSu(\mu)=\calF_0(X).
\]
The functions $\pqnSu{\,\cdot\,}$ are called the \emph{Sugeno-Lorentz prenorms}\/ on $\fSu(\mu)$
and the spaces $\fSu(\mu)$ equipped with them are called the \emph{Sugeno-Lorentz spaces}.
In particular, if $p=q=1$ then $(\,\cdot\,)_{1,1}$ is called the \emph{Sugeno prenorm}
and written by $(\,\cdot\,)_1$, that is, $(f)_1=(f)_{1,1}$ for every $f\in\calF_0(X)$.
The space $\fSu(\mu)$ equipped with the Sugeno prenorm is called the \emph{Sugeno space}.
In Section~\ref{Top} we will show that the pronorms
$\pqnSu{\,\cdot\,}$ and $(\,\cdot\,)_1$ generate the same topology
on $\fSu(\mu)$ and that those topologies coincide with the relative topology
induced from the topology generated by the Dunford-Schwartz prenorm $\dsn{\cdot}$
on $\calF_0(X)$ introduced in~\cite{K2021a}.
\begin{remark}\label{difference}
	 The Sugeno-Lorentz prenorm was defined in~\cite{K2021b} by
	\[
	\pqnSu{f}=\left(\frac{p}{q}\right)^{1/q}\Su(\mupq,|f|^q)^{1/q},
	\]
	which is not equivalent to formula~\eqref{New} in Definition~\ref{Lprenorm} since the Sugeno
	integral is not positively homogeneous.
	However, formula~\eqref{New} seems to be a  more appropriate definition of the Sugeno-Lorentz
	prenorm since it can reflect the feature that the Sugeno integral does not enjoy the positive homogeneity.
	Therefore, in this paper, we adopt formula~\eqref{New} as the definition of the Sugeno-Lorentz prenorm.
	As a matter of fact,
	it turns out that this change in definition does not have an essential effect on the results
	obtained in~\cite{K2021b} or the results of this paper, that is, with obvious modifications
	they can be shown for the Sugeno-Lorentz prenorm~\eqref{New} in the same way.
\end{remark}

\par

The properties of the Sugeno-Lorentz prenorm are collected in the following proposition,
all of which follow from easy calculation; see also~\cite{K2021a,K-Y,K2021b}.
\begin{proposition}\label{P-SuL}
Let $\mu\in\calM(X)$. Let $0<p<\infty$ and $0<q<\infty$.
\begin{enumerate}
\item[\us{(1)}] For any $f\in\fSu(\mu)$ it follows that
\[
\pqnSu{f}=\Su(\mu^{1/p},\ls|f|).
\]
\item[\us{(2)}] For any $A\in\calA$ and $\alpha\in\bbR$ it follows that
\[
\pqnSu{\alpha\chi_A}=\min\left\{|\alpha|\LS,\,\mu(A)^\frac{1}{p}\right\}.
\]
\item[\us{(3)}] For any $f\in\fSu(\mu)$
it follows that $\pqnSu{f}=0$ if and only if $\mu(\{|f|>c\})=0$ for every $c>0$;
they are equivalent to the condition that
$\mu(\{|f|>0\})=0$ if $\mu$ is null-continuous.
\item[\us{(4)}] For any $f\in\fSu(\mu)$ and $\alpha\in\bbR$
it follows that
\[
\pqnSu{\alpha f}\leq\max\left\{1,|\alpha|\right\}\pqnSu{f}.
\]
Hence the prenorm $\pqnSu{\,\cdot\,}$ is truncated subhomogeneous.
\item[\us{(5)}] For any $f\in\fSu(\mu)$
and $\ep>0$ it follows that
\[
\min\left\{\ep\left(\frac{p}{q}\right)^\frac{p}{q},\,\mu(\{|f|>\ep\})\right\}\leq\pqnSu{f}^p.
\]
\item[\us{(6)}] For any $f,g\in\fSu(\mu)$, if $|f|\leq |g|$ then $\pqnSu{f}\leq\pqnSu{g}$.
\item[\us{(7)}] $\mu$ is weakly null-additive if and only if $\pqnSu{\,\cdot\,}$
is weakly null-additive.
\item[\us{(8)}] $\mu$ is null-additive if and only if $\pqnSu{\,\cdot\,}$ is null-additive.
\item[\us{(9)}] $\mu$ is null-additive if and only if it follows that $\pqnSu{f}=\pqnSu{g}$
whenever $f,g\in\fSu(\mu)$ and $f\sim g$.
\item[\us{(10)}] If $\mu$ is $K$-relaxed subadditive for some $K\geq 1$,
then $\pqnSu{\,\cdot\,}$ satisfies $(2K)^{1/p}$-relaxed triangle inequality.
Furthermore, $\mu$ is subadditive if and only if the Sugeno prenorm $(\,\cdot\,)_1$
satisfies the triangle inequality.
\end{enumerate}
\end{proposition}
\par
From (4) and (10) of Proposition~\ref{P-SuL} it follows that
$\fSu(\mu)$ is a real linear subspace of $\calF_0(X)$ if $\mu$ is relaxed subadditive.

There is a close relationship between convergence in measure and convergence with respect to
$\pqnSu{\,\cdot\,}$.
The conclusion of the following proposition can be found in~\cite{W-K,W-M}, where $\mu$ is
assumed to satisfy a kind of continuity.
The same proof works for any nonadditive measures; see also~\cite{H-O1,H-O2}.
\begin{proposition}\label{CRS}
Let $\mu\in\calM(X)$. Let $0<p<\infty$ and $0<q<\infty$.
For any $\{f_n\}_\seqn\subset\calF_0(X)$ and $f\in\calF_0(X)$,
it follows that $\pqnSu{f_n-f}\to 0$ if and only if $f_n\mconv f$.
\end{proposition} 
The quotient space
\[
\qSu(\mu):=\{[f]\colon f\in\fSu(\mu)\}
\]
is defined by the equivalence relation introduced in Subsection~\ref{equiv}.
Given an equivalence class $[f]\in\qSu(\mu)$,
define the prenorm on $\qSu(\mu)$ by $\pqnSu{[f]}:=\pqnSu{f}$,
which is well-defined by (9) of Proposition~\ref{P-SuL},
provided that $\mu$ is null-additive.  
This prenorm has the same properties as the prenorm on $\fSu(\mu)$
and separates points of $\qSu(\mu)$, that is, for any $[f]\in\qSu(\mu)$,
if $\pqnSu{[f]}=0$ then $[f]=0$.
If $\mu$ is finite, then  $\qSu(\mu)=F_0(X)$.
\par
For any $a,b\in [0,\infty]$ and $0<r<\infty$ the following inequalities
\[
a\land b^r\leq (a\land b)^r+a\land b,\quad (a\land b)^r\leq (a\land b^r)^r+a\land b^r
\]
holds~\cite{H-O1}. The first inequality yields
\begin{equation}\label{EST1}
\pqnSu{f}\leq\left[\max\left\{1,\LS\right\}\nSu{f}\right]^\frac{1}{p}
+\max\left\{1,\LS\right\}\nSu{f}
\end{equation}
and the second one yields
\begin{equation}\label{EST2}
\nSu{f}^\frac{1}{p}\leq\left[\max\left\{1,\rLS\right\}\pqnSu{f}\right]^\frac{1}{p}
+\max\left\{1,\rLS\right\}\pqnSu{f}
\end{equation}
both of which hold for every $f\in\calF_0(X)$.
From these inequalities we see that for any sequence $\{f_n\}_\seqn\subset\calF_0(X)$,
it converges with respect to $\pqnSu{\,\cdot\,}$ if and only if
it converges with respect to $\nSu{\,\cdot\,}$
and that $\{\xi_n\}_\seqn$ is Cauchy with respect to $\pqnSu{\,\cdot\,}$ if and only if
it is Cauchy with respect to $\nSu{\,\cdot\,}$.
\par
The completeness and separability of the Sugeno-Lorentz spaces are already
investigated in~\cite{K2021b}.
For later use, in the following some results are extracted from~\cite{K2021b}
together with related terms.
Let $\mu\in\calM(X)$. We say that $\mu$ satisfies \emph{property (C)}\/ if for any
sequence $\{E_n\}_\seqn\subset\calA$, it follows that $\mu\left(\bigcup_{n=k}^\infty E_n\right)\to 0$
whenever $\sup_\seql\mu\left(\bigcup_{n=k}^{k+l}E_n\right)\to 0$.
Furthermore, we say that $\mu$ is
\emph{monotone autocontinuous from below}~\cite{Rebille}\/ if
$\mu(A\setminus B_n)\to\mu(A)$ whenever $A,B_n\in\calA$, $\mu(B_n)\to 0$,
and $\{B_n\}_\seqn$ is nonincreasing.
Every nonadditve measure that is continuous from below satisfies property~(C).
Other examples of nonadditive measures satisfying property~(C) can be found in~\cite[Proposition~3.3]{K-Y}.
The monotone autocontinuity from below obviously follows from the autocontinuity from below.
If $\mu$ is monotone autocontinuous from below, then it is null-additive.
\begin{theorem}\label{Su-comp}
Let $\mu\in\calM(X)$. Let $0<p<\infty$ and $0<q<\infty$.
Assume that $\mu$ is monotone autocontinuous from below and satisfies property~(C) and the (p.g.p.).
Then $\fSu(\mu)$ and $\qSu(\mu)$ are quasi-complete with respect to $\pqnSu{\,\cdot\,}$.
\end{theorem}
\begin{corollary}\label{Cor:Su-comp}
Let $\mu\in\calM(X)$. Let $0<p<\infty$ and $0<q<\infty$.
Assume that $\mu$ is relaxed subadditive, monotone autocontinuous from below,
and satisfies property~(C).
Then $\fSu(\mu)$ and $\qSu(\mu)$ are complete with respect to $\pqnSu{\,\cdot\,}$.
Furthermore, the prenorms $\pqnSu{\,\cdot\,}$ satisfy relaxed triangle inequalities.
\end{corollary}
\begin{remark}
Example~4.6 of~\cite{K2021b} shows that property~(C) cannot be dropped
in Theorem~\ref{Su-comp} and Corollary~\ref{Cor:Su-comp}.
\end{remark} 
The denseness and separability of the Sugeno-Lorentz spaces can be also characterized
in terms of some properties of nonadditive measures.
Recall that $\calS(X)$ is the set of all $\calA$-measurable simple functions on $X$
and $S(X):=\{[f]\colon f\in\calS(X)\}$.
\begin{theorem}\label{wLex}
Let $\mu\in\calM(X)$. Let $0<p<\infty$ and $0<q<\infty$.
Assume that $\mu$ is order continuous.
Then, $\calS(X)$ is a dense subset of $\fSu(\mu)$ with respect to $\pqnSu{\,\cdot\,}$.
If $\mu$ is additionally assumed to be null-additive,
then $S(X)$ is a dense subset of $\qSu(\mu)$.
\end{theorem}
%
%
%
We say that $\mu$ has a \emph{countable base}\/ if there is a countable subset $\calD$
of $\calA$ such that for any $A\in\calA$ and $\ep>0$ there is $D\in\calD$ for which
$\mu(A\triangle D)<\ep$.
\begin{theorem}
Let $\mu\in\calM(X)$. Let $0<p<\infty$ and $0<q<\infty$.
Assume that $\mu$ is order continuous and satisfies the (p.q.p.).
Assume that $\mu$ has a countable base.
Then there is a countable subset $\calE$ of $\fSu(\mu)$ such that for any $f\in\fSu(\mu)$
and $\ep>0$ there is $h\in\calE$ for which $\pqnSu{f-h}<\ep$.
Hence $\fSu(\mu)$ is separable with respect to $\pqnSu{\,\cdot\,}$.
If $\mu$ is additionally assumed to be null-additive, then $\qSu(\mu)$ is separable
with respect to $\pqnSu{\,\cdot\,}$.
\end{theorem}

%
%
%
\section{Topological properties of the Sugeno-Lorentz spaces}\label{Top}
%
%
%
Let $\mu\in\calM(X)$. For any $0<p<\infty$ and $0<q<\infty$,
let $\pqnSu{\,\cdot\,}$ be the Sugeno-Lorentz prenorm on $\fSu(\mu)$.
The Sugeno-Lorentz topology on $\fSu(\mu)$
is defined as the distance topology generated by the distance
\[
d(f,g):=\pqnSu{f-g}
\]
for every $f,g\in\fSu(\mu)$.
Let us recall the way of constructing the topology.
For each $f\in\fSu(\mu)$ and $r>0$, let
\[
S(f,r):=\{g\in\fSu(\mu)\colon\pqnSu{f-g}<r\}.
\]
For each $f\in\fSu(\mu)$, let $\calU(f)$ be the collection of all subsets $U$ of $\fSu(\mu)$
such that $S(f,r)\subset U$ for some $r>0$.
Let $\calT$ be the collection of all subsets $G$ of $\fSu(\mu)$
such that $G\in\calU(f)$ for every $f\in G$.
Then Theorem~\ref{Top_property} takes the following form.
\begin{theorem}\label{Su-Top}
The collection $\calT$ is a topology on $\fSu(\mu)$ and has the following properties.
\begin{enumerate}
\item[\us{(1)}] For any $f\in\fSu(\mu)$, the neighborhood system of $f$
is contained in $\calU(f)$.
\item[\us{(2)}] For any sequence $\{f_n\}_\seqn\subset\fSu(\mu)$ and $f\in\fSu(\mu)$, 
if $\pqnSu{f_n-f}\to 0$ then $f_n\to f$.
\item[\us{(3)}] The following conditions are equivalent.
\begin{enumerate}
\item[\us{(i)}] For any $f\in\fSu(\mu)$ and $r>0$, the sphere $S(f,r)$ is a neighborhood of $f$.
\item[\us{(ii)}] For any $f\in\fSu(\mu)$, the neighborhood system of $f$ coincides with $\calU(f)$.
\item[\us{(iii)}] The topology $\calT$ is first countable and
for any sequence $\{f_n\}_\seqn\subset\fSu(\mu)$ and $f\in\fSu(\mu)$, 
it follows that $f_n\to f$ if and only if $\pqnSu{f_n-f}\to 0$.
\end{enumerate}
\end{enumerate}
\end{theorem}
%
%
%
%
%
%
\begin{definition}\label{Def:Su-Top}
The topology $\calT$ on $\fSu(\mu)$ constructed in the way described above is called
the \emph{Sugeno-Lorentz topology generated by $\pqnSu{\,\cdot\,}$}.
The topological space $(\fSu(\mu),\calT)$ or its topology $\calT$
is called \emph{compatible with $\pqnSu{\,\cdot\,}$-convergence}\/ if one of equivalent
conditions (i)--(iii) of (3) in Theorem~\ref{Su-Top} is satisfied.
In terms of Proposition~\ref{CRS} it may be called \emph{compatible with convergence in $\mu$-measure}.
\end{definition}
Every sphere of any metric space is open with respect to the metric topology.
This is not the case for the Sugeno-Lorentz topology $\calT$ on $\fSu(\mu)$
as the following example shows.
\begin{example}\label{ex:open}
Let $X=\bbN$ and $\calA:=2^X$.
Let $\mu\colon\calA\to [0,\infty]$ be the nonadditive measure defined by
\[
\mu(A):=\begin{cases}
0 & \mbox{if }A=\emptyset,\\
|A|\cdot\sum\limits_{i\in A}\dfrac{1}{2^i} & \mbox{otherwise}
\end{cases}
\]
for every $A\in\calA$, where $|A|$ stands for the number of elements of $A$.
Let $A:=\{1\}$ and $B_n:=\{n+1\}$ for every $\seqn$.
For simplicity of notation, let $M_0:=\left(q/p\right)^{1/q}$.
Let $f:=M_0\chi_A$ and $f_n:=M_0\chi_{A\cup B_n}$ for every $\seqn$.
Then
\[
\pqnSu{f_n-f}
=\pqnSu{M_0\chi_{B_n}}=\min\left\{1,\mu(B_n)^\frac{1}{p}\right\}
=\min\left\{1,\frac{1}{2^{n+1}}\right\}\to 0,
\]
which implies $f_n\to f$ by (2) of Theorem~\ref{Su-Top}. In addition, we have
\[
\pqnSu{f_n}=\min\left\{1,\left(1+\frac{1}{2^n}\right)^\frac{1}{p}\right\}=1
\]
and
\[
\pqnSu{f}=\min\left\{1,\left(\frac{1}{2}\right)^\frac{1}{p}\right\}<1,
\]
so that $f_n\not\in S(0,1)$ and $f\in S(0,1)$.
Therefore, the sphere $S(0,1)$ is not open with respect to
the Sugeno-Lorentz topology $\calT$ on $\fSu(\mu)$.
\end{example}
The following theorem gives a necessary and sufficient condition that every sphere be open.
\begin{theorem}\label{OSC}
Let $\mu\in\calM(X)$. Let $0<p<\infty$ and $0<q<\infty$.
Let $\calT$ be the Sugeno-Lorentz topology on $\fSu(\mu)$ generated by $\pqnSu{\,\cdot\,}$.
\begin{enumerate}
\item[\us{(1)}] The following conditions are equivalent.
\begin{enumerate}
\item[\us{(i)}] $\mu$ is autocontinuous from above.
\item[\us{(ii)}] For every $f\in\fSu(\mu)$ and $r>0$ the sphere $S(f,r)$ is open with respect to $\calT$.
\end{enumerate}
Furthermore, if $\mu$ is autocontinuous from above, then
$\calU(f)$ is the neighborhood system of $f$ for every $f\in\fSu(\mu)$
and the topology $\calT$
is first countable.
\item[\us{(2)}] The following conditions are equivalent.
\begin{enumerate}
\item[\us{(i)}] $\mu$ is autocontinuous from below.
\item[\us{(ii)}] For every $f\in\fSu(\mu)$ and $r>0$,
the set $\{g\in\fSu(\mu)\colon\pqnSu{f-g}\leq r\}$ is closed with respect to $\calT$.
\end{enumerate}
\end{enumerate}
\end{theorem}
To prove Theorem~\ref{OSC} the following Fatou and reverse Fatou type lemmas
of the Sugeno integral are needed; see~\cite[Theorem~9.9]{W-K} for the proof.
\begin{proposition}\label{Fatou}
Let $\mu\in\calM(X)$.
\begin{enumerate}
\item[\us{(1)}] The following conditions are equivalent.
\begin{enumerate}
\item[\us{(i)}] $\mu$ is autocontinuous from below.
\item[\us{(ii)}] The Fatou convergence in measure lemma holds for $\mu$, that is,
for any sequence $\{f_n\}_\seqn\subset\calF_0^+(X)$ and $f\in\calF_0^+(X)$, if $f_n\mconv f$, then
it follows that
\[
\Su(\mu,f)\leq\liminf_\ninfty\Su(\mu,f_n).
\]
\end{enumerate}
\item[\us{(2)}] The following conditions are equivalent.
\begin{enumerate}
\item[\us{(i)}] $\mu$ is autocontinuous from above.
\item[\us{(ii)}] The reverse Fatou convergence in measure lemma holds for $\mu$, that is,
for any sequence $\{f_n\}_\seqn\subset\calF_0^+(X)$ and $f\in\calF_0^+(X)$, if $f_n\mconv f$, then
it follows that
\[
\limsup_\ninfty\Su(\mu,f_n)\leq\Su(\mu,f).
\] 
\end{enumerate}
\end{enumerate}
\end{proposition}
\noindent
\textbf{Proof of Theorem~\ref{OSC}}.
For simplicity of notation, let $L_0:=\left(p/q\right)^{1/q}$ and $M_0:=(q/p)^{1/q}$.
\par

(1)\ (i)$\Rightarrow$(ii)\ Suppose, contrary to our claim, that
$S(f_0,r_0)\not\in\calT$ for some $f_0\in\fSu(\mu)$ and $r_0>0$.
Then $S(f_0,r_0)\not\in\calU(g_0)$ for some $g_0\in S(f_0,r_0)$.
Hence, for each $\seqn$ there is $g_n\in S(g_0,1/n)$ such that $g_n\not\in S(f_0,r_0)$.
Since $\pqnSu{g_n-g_0}<1/n$ for every $\seqn$, we have $g_n\mconv g_0$
by Proposition~\ref{CRS}, hence $L_0|f_0-g_n|\mconv L_0|f_0-g_0|$.
Since $\mu^{1/p}$ is autocontinuous from above, it follows from Proposition~\ref{Fatou} that
\begin{align*}
\limsup_\ninfty\pqnSu{f_0-g_n}
&=\limsup_\ninfty\Su(\mu^{1/p},L_0|f_0-g_n|)\\
&\leq\Su(\mu^{1/p},L_0|f_0-g_0|)\\[1mm]
&=\pqnSu{f_0-g_0}<r_0
\end{align*}
and this contradicts to the fact that $g_n\not\in S(f_0,r_0)$ for every $\seqn$.
\par
(ii)$\Rightarrow$(i)\ Suppose, contrary to our claim, that there are
$A_0\in\calA$ and a sequence $\{B_n\}_\seqn\subset\calA$
such that $\mu(B_n)\to 0$ and $\mu(A_0\cup B_n)\not\to\mu(A_0)$.
If $\mu(A_0)=\infty$ we would see that $\mu(A_0\cup B_n)=\infty$ for every $\seqn$,
which contradicts to the fact that $\mu(A_0\cup B_n)\not\to\mu(A_0)$.
We thus assume that $\mu(A_0)<\infty$.
Then there are $\ep_0>0$ and a subsequence $\{B_{n_k}\}_\seqk$ of $\{B_n\}_\seqn$ such that
\begin{equation}\label{OSC1}
\mu(A_0\cup B_{n_k})>\mu(A_0)+\varepsilon_0
\end{equation}
for every $\seqk$.
Let $r_0:=(\mu(A_0)+\varepsilon_0)^\frac{1}{p}>0$.
Let $f_0:=M_0r_0\chi_{A_0}$ and $f_k:=M_0r_0\chi_{A_0\cup B_{n_k}}$ for every $\seqk$.
First calculate $\pqnSu{f_0}$ and have
\[
\pqnSu{f_0}=\min\left\{r_0,\mu(A_0)^\frac{1}{p}\right\}
=\mu(A_0)^\frac{1}{p}<r_0.
\]
Hence $f_0\in S(0,r_0)$. Next calculate $\pqnSu{f_k}$ and by~\eqref{OSC1} we have
\begin{align*}
\pqnSu{f_k}
&=\min\left\{r_0,\mu(A_0\cup B_{n_k})^\frac{1}{p}\right\}\\
&\geq\min\left\{r_0,(\mu(A_0)+\ep_0)^\frac{1}{p}\right\}=r_0.
\end{align*}
Hence $f_k\not\in S(0,r_0)$ for every $\seqk$.
Finally calculate $\pqnSu{f_k-f_0}$ and have
\[
\pqnSu{f_k-f_0}=\min\left\{r_0,\mu(B_{n_k}\setminus A_0)^\frac{1}{p}\right\}
\leq\mu(B_{n_k})^\frac{1}{p}.
\]
Therefore $\pqnSu{f_k-f_0}\to 0$ since $\mu(B_{n_k})\to 0$.
\par
Now, let us get a contradiction.
By assumption we have $S(0,r_0)\in\calT$, hence it follows from $f_0\in S(0,r_0)$ that
$S(0,r_0)\in\calU(f_0)$.
Therefore, $S(f_0,s_0)\subset S(0,r_0)$ for some $s_0>0$.
For this $s_0$, since $\pqnSu{f_k-f_0}\to 0$, we can find $k_0\in\bbN$
such that $\pqnSu{f_{k_0}-f_0}<s_0$.
Hence $f_{k_0}\in S(0,r_0)$.
This is impossible since $f_k\not\in S(0,r_0)$ for every $\seqk$.
The last statement follows from Theorem~\ref{Su-Top}.
\par
(2)\ (i)$\Rightarrow$(ii)\ For each $f\in\fSu(\mu)$ and $r>0$, let
\[
D(f,r):=\{g\in\fSu(\mu)\colon\pqnSu{f-g}\leq r\}
\]
and show that $D(f,r)^c\in\calT$.
Suppose, contrary to our claim, that $D(f_0,r_0)^c\not\in\calT$ for some $f_0\in\fSu(\mu)$ and $r_0>0$.
Then, by the same argument as the proof of (1)
one can find a sequence $\{g_n\}_\seqn\subset\fSu(\mu)$
and $g_0\in\fSu(\mu)$ such that $L_0|f_0-g_n|\mconv L_0|f_0-g_0|$
and $g_n\in D(f_0,r_0)$ for every $\seqn$.
Since $\mu^{1/p}$ is autocontinuous from below, it follows from Proposition~\ref{Fatou} that
\begin{align*}
\liminf_\ninfty\pqnSu{f_0-g_n}
&=\liminf_\ninfty\Su(\mu^{1/p},L_0|f_0-g_n|)\\
&\geq\Su(\mu^{1/p},L_0|f_0-g_0|)\\[1mm]
&=\pqnSu{f_0-g_0}>r_0
\end{align*}
and this contradicts to the fact that $g_n\in D(f_0,r_0)$ for every $\seqn$.
\par
(ii)$\Rightarrow$(i)\ Suppose, contrary to our claim, that there are $A_0\in\calA$
and a sequence $\{B_n\}\subset\calA$ such that $\mu(B_n)\to 0$
and $\mu(A_0\setminus B_n)\not\to\mu(A_0)$.
If $\mu(A_0)=0$ then we would see that $\mu(A_0\setminus B_n)=0$ for every $\seqn$,
which contradicts to the fact that $\mu(A_0\setminus B_n)\not\to\mu(A_0)$.
We thus assume that $\mu(A_0)>0$.
Then the proof falls naturally into two cases.
\par
First consider the case where $\mu(A_0)=\infty$. Then there are $r_0>0$ and
a sequence $\{B_{n_k}\}_\seqk$ of $\{B_n\}_\seqn$ such that
\begin{equation}\label{OSC2}
\mu(A_0\setminus B_{n_k})\leq r_0^p
\end{equation}
for every $\seqk$.
Let $f_0:=M_0(r_0+1)\chi_{A_0}$ and $f_k:=M_0(r_0+1)\chi_{A_0\setminus B_{n_k}}$ for every $\seqk$.
Then $\pqnSu{f_0}=r_0+1$, hence $f_0\not\in D(0,r_0)$.
Next, by~\eqref{OSC2} we have
\[
\pqnSu{f_k}=\mu(A_0\setminus B_{n_k})^\frac{1}{p}\leq r_0,
\]
so that $f_k\in D(0,r_0)$ for every $\seqk$.
Finally we have
\[
\pqnSu{f_k-f_0}=\min\left\{r_0+1,\mu(A_0\cap B_{n_k})^\frac{1}{p}\right\}
\leq\mu(B_{n_k})^\frac{1}{p},
\]
so that $\pqnSu{f_k-f_0}\to 0$ since $\mu(B_{n_k})\to 0$.
These observations lead to a contradiction by the same argument as the proof of (1).
\par
Next consider the case where $0<\mu(A_0)<\infty$.
Then there are $\ep_0>0$ and a subsequence $\{B_{n_k}\}_\seqk$ of $\{B_n\}_\seqn$ such that
\begin{equation*}\label{OSC3}
\mu(A_0)-\ep_0>\mu(A_0\setminus B_{n_k})
\end{equation*}
for every $\seqk$.
Let $r_0:=(\mu(A_0)-\ep_0)^\frac{1}{p}>0$.
Let $f_0:=M_0\mu(A_0)^\frac{1}{p}\chi_{A_0}$
and $f_k:=M_0\mu(A_0)^\frac{1}{p}\chi_{A_0\setminus B_{n_k}}$ for every $\seqk$.
Then, the same calculation as above
shows that $f_0\not\in D(0,r_0)$, $f_k\in D(0,r_0)$ for every $\seqk$,
and $\pqnSu{f_k-f_0}\to 0$. We thus get a contradiction.
\begin{definition}
We say that the topology $\calT$ satisfies the
\emph{open sphere condition}\/ if for any $f\in\fSu(\mu)$ and $r>0$ the sphere $S(f,r)$ is
open with respect to $\calT$.
\end{definition}
\par
At first glance, the Sugeno-Lorentz topology $\calT$ looks like depending on constants $p$ and $q$,
but actually it is only one regardless of $p$ and $q$.
This is the reason why this topology is written $\calT$ instead of $\calT_{p,q}$.
To verify this fact, first recall the Dunford-Schwartz topology on $\calF_0(X)$ introduced in~\cite{K2021a}.
\par
Let $\mu\in\calM(X)$.
Following~\cite[Definition~III.2.1]{D-S}, define the distance $\rho$ on $\calF_0(X)$ by
\[
\rho(f,g):=\inf_{c>0}\varphi\left(c+\mu(\{|f-g|>c\}\right)
\]
for every $f,g\in\calF_0(X)$, where the function $\varphi\colon [0,\infty]\to [0,\pi/2]$
is given by
\[
\varphi(t):=\begin{cases}
\arctan t & \mbox{if }t\ne\infty,\\
\pi/2 & \mbox{if }t=\infty.
\end{cases}
\]
Let
\[
\dsn{f}:=\rho(f,0)
\]
for every $f\in\calF_0(X)$,
which is called the \emph{Dunford-Schwartz prenorm}\/ on $\calF_0(X)$.
For each $f\in\calF_0(X)$ and $r>0$, let
\[
S_0(f,r):=\{g\in\calF_0(X)\colon\dsn{f-g}<r\}.
\]
For each $f\in\calF_0(X)$, let $\calU_0(f)$ be the collection of all subsets $U$ of $\calF_0(X)$
such that $S_0(f,r)\subset U$ for some $r>0$.
Let $\calT_0$ be the collection of all subsets $G$ of $\calF_0(X)$ such that $G\in\calU_0(f)$
for every $f\in G$.
Then, by~\cite[Theorem~4.1]{K2021a} the collection $\calT_0$
is a topology on $\calF_0(X)$ satisfying properties (1)--(3) of Theorem~\ref{Top_property}.
The topology $\calT_0$ is referred to as the \emph{Dunford-Schwartz topology}\/ on $\calF_0(X)$
generated by $\mu$.
\par
The Sugeno-Lorentz prenorm is closely related to the Dunford-Schwartz prenorm.
This relation is crucial to the proof of the fact that the Sugeno-Lorentz topology is independent
of $p$ and $q$. 
\begin{proposition}\label{IN}
Let $\mu\in\calM(X)$. Let $0<p<\infty$ and $0<q<\infty$.
\begin{enumerate}
\item[\us{(1)}] For any $f\in\calF_0(X)$ it follows that
\[
\dsn{f}\geq\min\left\{\varphi\left(\pqnSu{f}^p\right),
\varphi\left(\left(\frac{q}{p}\right)^\frac{1}{q}\pqnSu{f}\right)\right\}.
\]
\item[\us{(2)}] For any $f\in\calF_0(X)$ it follows that
\[
\dsn{f}\leq\varphi\left(\pqnSu{f}^p+\left(\frac{q}{p}\right)^\frac{1}{q}\pqnSu{f}\right).
\]
\end{enumerate}
\end{proposition}
\begin{proof}
For simplicity of notation, let $L_0:=(p/q)^{1/q}$ and $M_0:=(q/p)^{1/q}$.
\par
(1)\ The conclusion is obvious if $\pqnSu{f}=0$.
We thus assume that $\pqnSu{f}>0$.
Then, for any $c\in (0,M_0\pqnSu{f})$, we have
\[
\pqnSu{f}=\Su(\mu^{1/p},L_0|f|)\leq\min\left\{L_0c,\mu(\{|f|>c\})^\frac{1}{p}\right\},
\]
hence
\[
\mu(\{|f|>c\})\geq\pqnSu{f}^p,
\]
and finally
\begin{equation}\label{IN1}
\inf_{0<c<M_0\pqnSu{f}}\varphi(c+\mu(\{|f|>c\}))\geq\varphi\left(\pqnSu{f}^p\right).
\end{equation}
Meanwhile, we have
\begin{equation}\label{IN2}
\inf_{c\geq M_0\pqnSu{f}}\varphi(c+\mu(\{|f|>c\}))\geq\varphi\left(M_0\pqnSu{f}\right).
\end{equation}
Since
\begin{align*}
\dsn{f}
&=\min\left\{\inf_{0<c<M_0\pqnSu{f}}\varphi(c+\mu(\{|f|>c\})),\right.\\[1mm]
&\quad\qquad\qquad\qquad\left.\inf_{c\geq M_0\pqnSu{f}}\varphi(c+\mu(\{|f|>c\}))\right\},
\end{align*}
the desired inequality follows from~\eqref{IN1} and~\eqref{IN2}.
\par
(2)\ The conclusion is obvious if $\dsn{f}=0$.
Thus the proof naturally falls into two cases.
\par
First consider the case where $\dsn{f}=\pi/2$.
Then $\mu(\{|f|>c\})=\infty$ for every $c>0$, so that
$\pqnSu{f}=\infty$. This leads to the conclusion.
\par 
Next consider the case where $0<\dsn{f}<\pi/2$.
Since the funciton $\xi(t):=t^p+M_0t$ is a continuous and increasing function on $(0,\infty)$
satisfying $\lim_{t\to+0}\xi(t)=0$ and $\lim_{t\to\infty}\xi(t)=\infty$,
by the intermediate value theorem there is $c_0\in (0,\infty)$ such that
$\xi(c_0)=c_0^p+M_0c_0=\varphi^{-1}(\dsn{f})$.
For this $c_0>0$, since
\[
\dsn{f}\leq\varphi(M_0c_0+\mu(\{|f|>M_0c_0\}),
\]
it follows that
\[
\mu(\{L_0|f|>c_0\})^\frac{1}{p}\geq\left(\varphi^{-1}(\dsn{f})-M_0c_0\right)^\frac{1}{p}=c_0.
\]
Hence
\begin{align*}
\pqnSu{f}
&=\Su(\mu^{1/p},L_0|f|)\geq c_0\land\mu(\{L_0|f|>c_0\})^\frac{1}{p}\\
&=c_0=\left(\varphi^{-1}(\dsn{f})-M_0c_0\right)^\frac{1}{p},
\end{align*}
which yields
\begin{equation}\label{IN3}
\varphi^{-1}(\dsn{f})\leq\pqnSu{f}^p+M_0c_0,
\end{equation}
and finally
\begin{equation}\label{IN4}
c_0=\left(\varphi^{-1}(\dsn{f})-M_0c_0\right)^\frac{1}{p}\leq\pqnSu{f}.
\end{equation}
Hence, the desired inequality follows from~\eqref{IN3} and~\eqref{IN4}.
\end{proof}
\begin{theorem}\label{TopE}
Let $\mu\in\calM(X)$. Let $0<p<\infty$ and $0<q<\infty$.
The Sugeno-Lorentz topology $\calT$ on $\fSu(\mu)$ is the relative topology induced from
the Dunford-Schwartz topology $\calT_0$ on $\calF_0(X)$.
Hence, the Sugeno-Lorentz topology $\calT$ does not depend on $p$ and $q$.
\end{theorem}
\begin{proof}
For simplicity of notation, let $M_0:=(q/p)^{1/q}$.
We first show that $\calT\subset\calT_0\cap\fSu(\mu)$.
Let $H\in\calT$ and $f\in H$. Then $S(f,r)\subset H$ for some $r>0$.
Let
\[
r_0:=\min\left\{\varphi\left(r^p\right),\varphi\left(M_0r\right)\right\}.
\]
Then we must have $S_0(f,r_0)\subset S(f,r)$, for otherwise there is $g_0\in S_0(f,r_0)$ but
$g_0\not\in S(f,r)$.
Since $\pqnSu{f-g_0}\geq r$, by (1) of Proposition~\ref{IN} we have
\[
\dsn{f-g_0}\geq\left\{\varphi\left(r^p\right),\varphi\left(M_0r\right)\right\}=r_0,
\]
a contradiction. Consequently, $H\in\calT_0$.
Since $H$ is a subset of $\fSu(\mu)$, we conclude that $H\in\calT_0\cap\fSu(\mu)$.
\par
Next we show that $\calT_0\cap\fSu(\mu)\subset\calT$. To this end,
let $H\in\calT_0\cap\fSu(\mu)$.
Then $H=G\cap\fSu(\mu)$ for some $G\in\calT_0$.
Let $f\in H$. Since $f\in G$, it follows that $S_0(f,r)\subset G$ for some $r>0$.
Since the function $\xi(t):=t^p+M_0t$ is a continuous and increasing function on $(0,\infty)$
satisfying $\lim_{t\to+0}\xi(t)=0$ and $\lim_{t\to\infty}\xi(t)=\infty$,
by the intermediate value theorem there is $r_0\in (0,\infty)$
such that $\xi(r_0)=\varphi^{-1}(r)$, hence
\begin{equation}\label{TopE1}
\varphi\left(r_0^p+M_0r_0\right)=r.
\end{equation}
Then we must have $S(f,r_0)\subset S_0(f,r)$, for otherwise there is $g_0\in S(f,r_0)$ but
$g_0\not\in S_0(f,r)$. Since $\pqnSu{f-g_0}<r_0$, by (2) of Proposition~\ref{IN} we have
\[
\dsn{f-g_0}<\varphi\left(r_0^p+M_0r_0\right)=r,
\]
a contradiction.
It thus follows that $S(f,r_0)\subset H$, hence $H\in\calT$.
\end{proof}
From Theorem~\ref{TopE} we see that the topologies on $\fSu(\mu)$
generated by the Sugeno-Lorentz prenorm
$\pqnSu{\,\cdot\,}$ and by the Sugeno prenorm $(\,\cdot\,)_1$ are equal to each other
and that they coincide with the relative topology induced from the Dunford-Schwartz topology
on $\calF_0(X)$.
\par

In the rest of this section, assume that $\mu$ is null-additive in such a way that the quotient
space $\qSu(\mu)$ and the prenorm $\pqnSu{\,\cdot\,}$ on $\qSu(\mu)$ are well-defined.
Let $\tcalT$ be the distance topology generated by the distance
\[
\td([f],[g]):=\pqnSu{[f]-[g]}
\]
for every $[f],[g]\in\qSu(\mu)$.
In other words, $\tcalT$ is defined in the following way:
For each $[f]\in\qSu(\mu)$ and $r>0$, let
\[
\tS([f],r):=\{[g]\in\qSu(\mu)\colon\pqnSu{[f]-[g]}<r\}.
\]
For each $[f]\in\qSu(\mu)$, let $\tcalU([f])$ be the collection of all subsets $\tU$ of $\qSu(\mu)$
such that $\tS([f],r)\subset\tU$ for some $r>0$.
Let $\tcalT$ be the collection of all subsets $\tG$ of $\qSu(\mu)$
such that $\tG\in\tcalU([f])$ for every $[f]\in\tG$.
Then Theorems~\ref{Top_property} and~\ref{OSC} take the following form.
\begin{theorem}\label{qSu-Top}
The collection $\tcalT$ is a topology on $\qSu(\mu)$ and has the following properties.
\begin{enumerate}
\item[\us{(1)}] For any $[f]\in\qSu(\mu)$, the neighborhood system of $[f]$
is contained in $\tcalU([f])$.
\item[\us{(2)}] For any sequence $\{[f_n]\}_\seqn\subset\qSu(\mu)$ and $[f]\in\qSu(\mu)$, 
if $\pqnSu{[f_n]-[f]}\to 0$ then $[f_n]\to [f]$.
\item[\us{(3)}] The following conditions are equivalent.
\begin{enumerate}
\item[\us{(i)}] For any $[f]\in\qSu(\mu)$ and $r>0$, the sphere $\tS([f],r)$
is a neighborhood of $[f]$.
\item[\us{(ii)}] For any $[f]\in\qSu(\mu)$, the neighborhood system of $[f]$
coincides with $\tcalU([f])$.
\item[\us{(iii)}] The topology $\tcalT$ is first countable and for any sequence
$\{[f_n]\}_\seqn\subset\qSu(\mu)$ and $[f]\in\qSu(\mu)$, 
it follows that $[f_n]\to [f]$ if and only if $\pqnSu{[f_n]-[f]}\to 0$.
\end{enumerate}
\item[\us{(4)}] The following conditions are equivalent.
\begin{enumerate}
\item[\us{(i)}] $\mu$ is autocontinuous from above.
\item[\us{(ii)}] For every $[f]\in\qSu(\mu)$ and $r>0$ the sphere $\tS([f],r)$
is open with respect to $\tcalT$.
\end{enumerate}
Furthermore, if $\mu$ is autocontinuous from above, then $\tcalU([f])$ is the
neighborhood system of $[f]$ for every $[f]\in\qSu(\mu)$ and the topology $\tcalT$
is first countable.
\item[\us{(5)}] The following conditions are equivalent.
\begin{enumerate}
\item[\us{(i)}] $\mu$ is autocontinuous from below.
\item[\us{(ii)}] For every $[f]\in\qSu(\mu)$ and $r>0$
the set $\{[g]\in\qSu(\mu)\colon\pqnSu{[f]-[g]}\leq r\}$ is closed with respect to $\tcalT$.
\end{enumerate} 
\end{enumerate}
\end{theorem}
\begin{definition}
The topological space $(\qSu(\mu),\tcalT)$ or its topology $\tcalT$ is called
\emph{compatible with $\pqnSu{\,\cdot\,}$-convergence} if one of equivalent
conditions (i)--(iii) of (3) in Theorem~\ref{qSu-Top} is satisfied.
In terms of Proposition~\ref{CRS}
it may be called \emph{compatible with convergence in $\mu$-measure}.
Furthermore, if for any $[f]\in\qSu(\mu)$ and $r>0$ the sphere $\tS([f],r)$ is open
with respect to $\tcalT$, then we say that the topology $\tcalT$
satisfies the \emph{open sphere condition}.
\end{definition}

\begin{theorem}\label{Qtop}
Let $\kappa\colon\fSu(\mu)\to\qSu(\mu)$ be the quotient map defined by $\kappa(f):=[f]$
for every $f\in\fSu(\mu)$.
Then $\tcalT$ is the collection of all subsets $\tG$ of $\qSu(\mu)$
such that $\kappa^{-1}(\tG)\in\calT$.
Therefore, $\tcalT$ is the quotient topology induced on $\qSu(\mu)$ by $\calT$ and $\kappa$.
Furthermore, $\kappa$ is continuous.
\end{theorem}
\begin{proof}
Let $\tG\in\tcalT$.
For any $f\in\kappa^{-1}(\tG)$, we have $[f]=\kappa(f)\in\tG$, so that
$\tS([f],r)\subset\tG$ for some $r>0$.
Hence it follows that $S(f,r)=\kappa^{-1}(\tS([f],r))\subset\kappa^{-1}(\tG)$,
and hence, $\kappa^{-1}(\tG)\in\calT$.
Conversely, let $\tG$ be a subset of $\qSu(\mu)$ such that $\kappa^{-1}(\tG)\in\calT$.
For any $[f]\in\tG$, we have $f\in\kappa^{-1}(\tG)$,
so that $S(f,r)\subset\kappa^{-1}(\tG)$ for some $r>0$.
Then it follows that $\tS([f],r)=\kappa(S(f,r))\subset\tG$, and hence, $\tG\in\tcalT$.
The fact that $\kappa$ is continuous is easy to show.
\end{proof}
From Theorems~\ref{TopE} and~\ref{Qtop}
it is also seen that the Sugeno-Lorentz topology $\tcalT$ on $\qSu(\mu)$
does not depend on $p$ and $q$ and that it is only one regardless of $p$ and $q$.
\par
It is desirable that the topologies on $\fSu(\mu)$ and $\qSu(\mu)$ contain enough open sets
to distinguish their points in some way.
This is accomplished when the topologies satisfy the Hausdorff separation axiom.
\begin{theorem}\label{T2}
Let $\mu\in\calM(X)$. Let $0<p<\infty$ and $0<q<\infty$.
\begin{enumerate}
\item[\us{(1)}] Assume that $\mu$ is autocontinuous from below or satisfies the (p.g.p.).
Then, for any $f,g\in\fSu(\mu)$, if $\pqnSu{f-g}>0$ then $S(f,r)\cap S(g,r)=\eset$
for some $r>0$.
\item[\us{(2)}] Assume that $\mu$ is autocontinuous from above.
If $\mu$ is autocontinuous from below or satisfies the (p.g.p.), then the topological space
$(\qSu(\mu),\tcalT)$ is Hausdorff.
\end{enumerate}
\end{theorem}
\begin{proof}
(1)\ Let $f,g\in\fSu(\mu)$ and assume that $\pqnSu{f-g}>0$.
Suppose, contrary to our claim, that $S(f,r)\cap S(g,r)\ne\eset$ for every $r>0$.
Then there is $\{h_n\}_\seqn\subset\fSu(\mu)$ such that $\pqnSu{h_n-f}\to 0$
and $\pqnSu{h_n-g}\to 0$.
Hence $h_n\mconv f$ and $h_n\mconv g$ by Proposition~\ref{CRS}.
Let $c>0$. Let $A_n:=\{|h_n-f|>c/2\}$ and $B_n:=\{|h_n-g|>c/2\}$ for every $\seqn$.
Then $\mu(A_n)\to 0$ and $\mu(B_n)\to 0$.
In addition, for any $\seqn$,
\begin{equation}\label{T2-1}
\{|f-g|>c\}\setminus A_n\subset B_n
\end{equation}
and
\begin{equation}\label{T2-2}
\{|f-g|>c\}\subset A_n\cup B_n.
\end{equation}
\par
If $\mu$ is autocontinuous from below, then
\[
\mu(\{|f-g|>c\}\setminus A_n)\to\mu(\{|f-g|>c\}),
\]
so that
\begin{equation}\label{T2-3}
\mu(\{|f-g|>c\})=0
\end{equation}
follows from~\eqref{T2-1}.
If $\mu$ satisfies the (p.g.p.), then~\eqref{T2-3} follows from~\eqref{T2-2}.
Since $c>0$ is arbitrary, by~\eqref{T2-3} we have $\pqnSu{f-g}=0$, a contradiction.
\par
(2)\ Let $[f],[g]\in\qSu(\mu)$ and assume that $[f]\ne [g]$.
Then $\pqnSu{f-g}>0$, hence assertion~(1) of this theorem
implies that $S(f,r)\cap S(g,r)=\eset$ for some $r>0$.
Let $\tU:=\tS([f],r)$ and $\tV:=\tS([g],r)$. Since $\mu$ is autocontinuous from above,
it follows from Theorem~\ref{qSu-Top} that
$\tU$ and $\tV$ are open neighborhoods of $[f]$ and $[g]$, respectively.
The disjointness of $\tU$ and $\tV$ follows from the fact that $S(f,r)\cap S(g,r)=\eset$.
\end{proof}
Metrizability of a topology is also one of important issues in the study of
topological spaces.
Recall that a topological space is (pseudo)metrizable if the space can be given a (pseudo)metric
such that the distance topology generated by the (pseudo)metric coincides with the given topology
on the space.
In this case, the given topology is called \emph{(pseudo)metrizable by the (pseudo)metric}. 
%
%
%
\begin{theorem}\label{Metrizable}
Let $\mu\in\calM(X)$.
Assume that $\mu$ is subadditive.
Then the topology $\calT$ is pseudometrizable by the pseudometric $d$ on $\fSu(\mu)$
defined by $d(f,g):=(f-g)_1$ for every $f,g\in\fSu(\mu)$, while $\tcalT$ is metrizable
by the metric $\td$ on $\qSu(\mu)$ defined by $\td([f],[g]):=([f]-[g])_1$
for every $[f],[g]\in\qSu(\mu)$.
\end{theorem}
\begin{proof}
It follows from Proposition~\ref{K-sub} that  $d$ is a pseudometric on $\fSu(\mu)$,
while $\td$ is a metric on $\qSu(\mu)$.
Furthermore, by Theorems~\ref{TopE} and~\ref{Qtop} the topologies $\calT$ and $\tcalT$
are generated by $d$ and $\td$, respectively.
\end{proof}
\begin{remark}
If $\mu$ is $K$-relaxed subadditive for some $K\geq 1$,
then $d$ and $\td$ satisfy the $K$-relaxed triangle inequality
and the topologies $\calT$ and $\tcalT$ are generated by $d$ and $\td$, respectively.
\end{remark}
\section{Topological linear properties of the Sugeno-Lorentz spaces}\label{TLS}
In this section, basic topological linear properties of the Sugeno-Lorentz spaces are
deduced from the corresponding properties of the space $\calF_0(X)$ developed in~\cite{K2021a}.
\begin{theorem}\label{TVS}
Let $\mu\in\calM(X)$. Assume that $\mu$ is relaxed subadditive.
If $\mu$ is order continuous and autocontinuous from above, then
$(\fSu(\mu),\calT)$ and $(\qSu(\mu),\tcalT)$ are topological linear spaces satisfying 
the open sphere condition.
Furthermore, $\tcalT$ is completely regular.
\end{theorem}
\begin{proof}
By assumption, $\mu$ is even null-additive and satisfies the (p.g.p.).
It thus follows from~\cite[Theorem~9.3]{K2021a} that $(\calF_0(X),\calT_0)$ is a
topological linear space.
Note that the relaxed subadditivity of $\mu$ implies that
$\fSu(\mu)$ is a linear subspace of $\calF_0(X)$.
Therefore, it follows form~\cite[page 17]{Schaefer} that $(\fSu(\mu),\calT)$ is a topological linear space
since $\calT$ is the relative topology induced from $\calT_0$ by Theorem~\ref{TopE}.
The open sphere condition of $\calT$ follows from Theorem~\ref{OSC}.
\par
Next, since $\tcalT$ is the quotient topology induced by $\calT$ and $\kappa$
by Theorem~\ref{Qtop}, it follows from~\cite[page 20]{Schaefer}
that $(\qSu(\mu),\tcalT)$ is a topological linear space.
The open sphere condition of $\tcalT$ follows from Theorem~\ref{qSu-Top}.
\par
Finally, the linear topology $\tcalT$ is Hausdorff by Theorem~\ref{T2}.
It thus follows from~\cite[page~16]{Schaefer} that $\tcalT$ is completely regular.
\end{proof}
\begin{remark}
Theorem~\ref{TVS} contains \cite[Corollary~2]{H-O2} since every finite nonadditive measure
that is weakly subadditive and continuous from above in the sense of~\cite{H-O2} is relaxed subadditive,
order continuous, and autocontinuous from above.
\end{remark} 
%
%
%
%
%
%
%
%
%
The following corollary is an immediate consequence of (4) of Proposition~\ref{P-SuL}
and Theorem~\ref{Metrizable}.
%
%
%
\begin{corollary}
Let $\mu\in\calM(X)$.
Assume that $\mu$ is order continuous and subadditive.
Then the topological linear space $(\fSu(\mu),\calT)$ is pseudometrizable
by the pseudometric $d$ on $\fSu(\mu)$
defined by $d(f,g):=(f-g)_1$ for every $f,g\in\fSu(\mu)$,
while $(\qSu(\mu),\tcalT)$ is metrizable by the metric $\td$ on $\qSu(\mu)$
defined by $\td([f],[g]):=([f]-[g])_1$ for every $[f],[g]\in\qSu(\mu)$.
Furthermore, $d$ and $\td$ are translation-invariant and truncated subhomogeneous.
\end{corollary}
\section{Summary of results}\label{conclusion}
In this paper, given a nonadditive measure $\mu$ and constants $0<p<\infty$ and $0<q<\infty$,
the Sugeno-Lorentz space $\fSu(\mu)$
and its quotient space $\qSu(\mu)$ are defined by using the Sugeno integral,
together with the Sugeno-Lorentz topologies $\calT$ on $\fSu(\mu)$ and $\tcalT$ on $\qSu(\mu)$
generated by $\pqnSu{\,\cdot\,}$.
As a continuation of our research~\cite{K2021b}
on the completeness and separability of the spaces,
this paper focuses on their topological and topological linear properties.
Some of our results are as follows.
\begin{itemize}
\item Any sphere defined by the Sugeno-Lorentz prenorm $\pqnSu{\,\cdot\,}$
is open with respect to the Sugeno-Lorentz topology
if and only if $\mu$ is autocontinuous from above.
\item The Sugeno-Lorentz topology on $\fSu(\mu)$
is the relative topology induced from the Dunford-Schwartz topology on $\calF_0(X)$.
Consequently, the Sugeno-Lorentz topology does not depend on $p$ and $q$,
and hence, it is only one regardless of $p$ and $q$.
\item Assume that $\mu$ is autocontinuous from above.
Then, the Sugeno-Lorentz topology on $\qSu(\mu)$ is Hausdorff if $\mu$ is autocontinuous form
below or satisfies the (p.g.p.).
\item Assume that $\mu$ is subadditive. Then the Sugeno-Lorentz topology on $\fSu(\mu)$ is
pseudometrizable, while the Sugeno-Lorentz topology on $\qSu(\mu)$ is metrizable.
\end{itemize}
\par
As the basic topological linear properties of the Sugeno-Lorentz spaces, the following results
are shown.
\begin{itemize}
\item Assume that $\mu$ is relaxed subadditive.
If $\mu$ is order continuous and autocontinuous from above, then the Sugeno-Lorentz spaces
$\fSu(\mu)$ and $\qSu(\mu)$ are topological linear spaces satisfying the open sphere condition.
\item If $\mu$ is order continuous and subadditive, then $(\fSu(\mu),\calT)$
is a pseudometrizable topological linear space, while $(\qSu(\mu),\tcalT)$ is a metrizable
topological linear space.
\end{itemize}
\par
Finally it should be remarked that
some of the foregoing results are peculiar to the Sugeno integrable function spaces.
A proper nonadditive counterpart of ordinary Lorentz spaces is certainly the Choquet-Lorentz spaces
defined by the Choquet integral; see~\cite{K-Y}.
%
%

\end{document}